\documentclass[11pt,
]{article}
\pdfoutput=1
\usepackage[english]{babel}
\usepackage{amsthm,amsmath,amssymb,amsfonts}
\usepackage[alphabetic]{amsrefs}
\usepackage{mathtools} % better equation output than original ams
\usepackage{thm-restate}
\usepackage{graphicx}
\usepackage[colorlinks=false]{hyperref}
\usepackage{setspace}
\makeatletter
\def\th@remark{%
  \thm@headfont{\bfseries}%
  \normalfont % body font
  \thm@preskip\topsep \divide\thm@preskip\tw@
  \thm@postskip\thm@preskip
}
\makeatother
\newtheorem{theorem}{Theorem}[section]
\newtheorem{corollary}[theorem]{Corollary}
\newtheorem{lemma}[theorem]{Lemma}
\theoremstyle{remark}
\newtheorem*{remark}{Remark}
\newtheorem*{question}{Question}

\theoremstyle{definition}

\newtheorem{innercustomcor}{Corollary}
\newenvironment{customcor}[1]
{\renewcommand\theinnercustomcor{#1}\innercustomcor}
{\endinnercustomcor}
\newcommand{\ds}{\displaystyle}

\newcommand{\Z}{\ensuremath{\mathbb{Z}}}

\newcommand{\V}{\ensuremath{\mathcal{V}}}

\newcommand{\Addresses}{{% additional braces for segregating \footnotesize
  \bigskip
  \footnotesize
  \textsc{
Dept. of Mathematics \& Statistics, McMaster University\\
1280 Main Street West
Hamilton, Ontario,
Canada L8S 4K1
  }\par\nopagebreak
      \textit{E-mail address}: \texttt{chenj293@mcmaster.ca}
  }}
\usepackage{color}
\title{On the algebraic Gordian distance}
\author{Jie CHEN
}
\date{}
\begin{document}
%-=-=-=-=-=-=-=-=-=-=-=-=-=-=-=-=-=-=-=-=-=-=-=-=-=-=-=-=-=-=-=-=-=-=-
%-=-=-=-=-=-=-=-=-=-=-=-=-=-=-=-=-=-=-=-=-=-=-=-=-=-=-=-=-=-=-=-=-=-=-
%-=-=-=-=-=-=-=-=-=-=-=-=-=-=-=-=-=-=-=-=-=-=-=-=-=-=-=-=-=-=-=-=-=-=-
%-=-=-=-=-=-=-=-=-=-=-=-=-=-=-=-=-=-=-=-=-=-=-=-=-=-=-=-=-=-=-=-=-=-=-
\newcounter{minutes}\setcounter{minutes}{\time}
\divide\time by 60
\newcounter{hours}\setcounter{hours}{\time}
\multiply\time by 60
\addtocounter{minutes}{-\time}
\def\thefootnote{}
\footnotetext{
\texttt{\tiny File:~\jobname .tex, printed: \number\year-\number\month-\number\day,
          \thehours:\ifnum\theminutes<10{0}\fi\theminutes}
}
\makeatletter
\def\thefootnote{\@arabic\c@footnote}
\newcommand{\subjclass}[2][1991]{%
  \let\@oldtitle\@title%
  \gdef\@title{\@oldtitle\footnotetext{#1 \emph{Mathematics subject classification.} #2}}%
}
\newcommand{\keywords}[1]{%
  \let\@@oldtitle\@title%
  \gdef\@title{\@@oldtitle\footnotetext{\emph{Key words and phrases.} #1.}}%
}
\makeatother
%-=-=-=-=-=-=-=-=-=-=-=-=-=-=-=-=-=-=-=-=-=-=-=-=-=-=-=-=-=-=-=-=-=-=-
%-=-=-=-=-=-=-=-=-=-=-=-=-=-=-=-=-=-=-=-=-=-=-=-=-=-=-=-=-=-=-=-=-=-=-
%-=-=-=-=-=-=-=-=-=-=-=-=-=-=-=-=-=-=-=-=-=-=-=-=-=-=-=-=-=-=-=-=-=-=-
%-=-=-=-=-=-=-=-=-=-=-=-=-=-=-=-=-=-=-=-=-=-=-=-=-=-=-=-=-=-=-=-=-=-=-
\maketitle
\begin{abstract}
%To find restrictions that two $S$-equivalence classes should 
%bear when their algebraic Gordian distance is one,
%We construct the Blanchfield pairings
%of two Seifert matrices mutually convertible by an algebraic unknotting operation.
Using Blanchfield pairings, 
%we improve a theorem of Kawauchi.
%Our results show
we show
that two Alexander polynomials cannot be 
realized by  a pair of matrices with Gordian distance one if a corresponding
quadratic equation does not have
an
integer solution.
We also give an example of how our results help in calculating the Gordian distances,
algebraic Gordian distances and polynomial distances.
\end{abstract}
\section{Introduction}
A \emph{knot} is an oriented circle embedded in the three-sphere $S^3$,  taken up to isotopy, 
and a knot diagram is an oriented circle immersed in $S^2$ with at worst double points, which are crossings and where one records over and under crossing information.
Any knot diagram can be converted to a diagram for the trivial knot by crossing changes, and thus crossing change is an unknotting operation for knots. 
%A crossing change on a knot is called the \emph{unknotting operation}, which is a geometric operation.
%Using crossing changes, any knot can be converted to the unknot and so crossing change  is an unknotting operation for knots.
The Gordian distance $d_G(K,K')$ between two knots $K$ and $K'$ is the minimal number of crossing changes needed to turn $K$ into $K'$, 
and the unknotting number of $K$ is defined by $u(K)=d_G(K,O)$, where $O$ is the trivial knot.

Murakami defined the algebraic unknotting operation in \cite{mura90} in terms of the corresponding Seifert matrices. 
Let $V$ be a Seifert matrix, and let $[V]$ be its $S$-equivalence class.
The algebraic Gordian distance $d^a_G([V],[V'])$ between two $S$-equivalence classes of Seifert matrices
is the minimal number of algebraic unknotting operations 
needed to turn a Seifert matrix in $[V]$ into a Seifert matrix in $[V']$.
The algebraic unknotting number is defined by $u_a([V])=d^a_G([V],[\varnothing])$, where
$[\varnothing]$ denotes the $S$-equivalence class of the empty matrix.

Seifert showed that a Laurent polynomial $\Delta \in \Z[t,t^{-1}]$ is the Alexander polynomial of knot if and only if $\Delta(t^{-1}) = \Delta(t)$ and $\Delta(1)=1$; see \cite{seifert-a}. Given two such polynomials $\Delta, \Delta'$, Kawauchi defined the distance between them by setting
$$\ds    \rho(\Delta,\Delta')=\min_{K,K'} d_G(K,K'),$$
where $K$ and $K'$ are knots with Alexander polynomials $\Delta_K =\Delta$
and $\Delta_{K'} = \Delta'$; see \cite{kawa12}.
It is known that $\rho(\Delta,\Delta')\in \left\{ 0,1,2 \right\}$, and that 
$\rho(\Delta,\Delta')=0\Leftrightarrow \Delta=\Delta'$.
There are many examples of Alexander polynomials $\Delta$, $\Delta'$ with 
$\rho(\Delta,\Delta')=1$. 
Jong posed the problem of finding examples of Alexander polynomials $\Delta$ and $\Delta'$ with
$\rho(\Delta,\Delta')= 2$; see \cites{kawa12,jong1,jong2,jong3}.

Let $\Lambda$ be the Laurent polynomial ring $\Z[t,t^{-1}]$.
For a polynomial $c\in\Lambda$, put $\bar{c}=c|_{t=t^{-1}}$.
%Let $\Delta_K$ and $\Delta_{K'}$ denote the Alexander polynomials of $K$ and $K'$ respectively.
In \cite{kawa12}*{Theorem~1.2, p.949} Kawauchi  proved the following theorem:
\begin{theorem}[Kawauchi]  
    If $u(K)=d_G(K,K')=1$, then there exists $c\in\Lambda $ such that 
    $\pm \Delta_{K'}\equiv c\bar{c}\pmod{\Delta_K}$.
    \label{thm:kawa}
\end{theorem}
In this paper, by considering the corresponding algebraic unknotting operations,
we obtain the following result.
\begin{customcor}{\ref{col:alg}}
    If $u_a([V])=d_G^a([V],[V'])=1$,
then there exists $c\in \Lambda$ such that 
$\pm \Delta_{V'}\equiv c\bar{c}\pmod{\Delta_V}$.
\end{customcor}
The proof of Corollary~\ref{col:alg} will be given in Section~4. Note that
Corollary~\ref{col:alg} implies
Theorem~\ref{thm:kawa}, for if $K$ and $K'$ are knots with Seifert matrices $V$ and $V'$, respectively, then $d_G(K,K')=1$ implies $d_G^a([V],[V'])\le 1$. Note further that
the converse may not hold, and that Corollary~\ref{col:alg} does not require any geometric assumption on the unknotting number.
Further, Corollary~\ref{col:alg} gives some new results that
cannot be derived from Theorem~\ref{thm:kawa}.
For instance, the following corollary gives an obstruction to
$d^a_G([V],[V'])=1$.
\begin{customcor}{\ref{thm_main3}}
    Suppose that $V$ and $V'$ are Seifert matrices with Alexander polynomials
    $\Delta_V$ and $\Delta_{V'}$, respectively.
    Suppose further that $\Delta_V=h(t+t^{-1})+1-2h$ for $|h|$ a prime or $1$ and that
    $\Delta_{V'}\equiv d \pmod{\Delta_{V}}$, where $0 \neq d\in \Z$. 
    If $u_a([V])=1$ and if
    $h^2x^2+y^2+(2h-1)xy=\pm d $ admits no integer solutions,
    then the algebraic Gordian distance $d^a_G([V],[V'])> 1$.
\end{customcor}
The following corollary,
also derived from Corollary~\ref{col:alg}, gives a new solution to Jong's problem.
\begin{customcor}{\ref{thm_main4}}
    The Alexander polynomial distance
    $\rho(t-1+t^{-1},\Delta)=2$ if
$\Delta\equiv  2+4m \pmod{t-1+t^{-1}}$
for some $m\in \Z$.
\end{customcor}
In Section~5, we apply Corollary~\ref{thm_main4} to show
that the Gordian distance $ d_G(K_1,K_2)\ge 2$
for any pair of knots $K_1$ and $K_2$ with
$\Delta_{K_{1}}=\Delta_{3_{1}}$ and $\Delta_{K_{2}}=\Delta_{9_{25}}$. 
% This result cannot be derived from Theorem~\ref{thm:kawa} due 
% to the geometric assumption of the unknotting number.

The remainder of this paper is organized as follows.
In Section~2, we review preliminary material.
In Section~3, we present some results on Seifert matrices.
In Section~4,
we prove Theorem \ref{thm_main2}, our main result, which 
is an improvement to Theorem \ref{thm:kawa} and gives new answers to Jong's problem.
%In addition, our result provides a new method to answer Nakanishi's question.
In the final section, we present an example illustrating how to calculate 
various distances in knot theory.

\section{Preliminaries}

A \emph{Seifert matrix} $V$
is a $2n\times 2n$ integer matrix, satisfying $\det (V-V^T)=1$.
Two Seifert matrice $V$ and $W$ are said to be \emph{congruent}
if $W=PVP^T$ for a unimodular matrix 
    $P$. A Seifert matrix $W$ is called an \emph{enlargement} of $V$ if
\[
    W=   \begin{pmatrix}
        0 &0 &0\\
        1 &x &M\\
        0 &N^T &V
    \end{pmatrix}
    \text{~~~~~~~~or~~~~~~~~}
   \begin{pmatrix}
        0 &1 &0\\
        0 &x &M\\
        0 &N^T &V
    \end{pmatrix},
\]
where $M$ and $N$ are row vectors. 
    In this case, we also say that $V$ is a \emph{reduction} of $W$.
    Two Seifert matrices are said to be \emph{$S$-equivalent} if one can be obtained from the other by a sequence of congruences, enlargements, and reductions. Note that any two Seifert matrices of the same knot are $S$-equivalent \cite{levine}.
    For a given
    Seifert matrix $V$, we use $[V]$ to denote its
    \emph{$S$-equivalence class,} which consists of all Seifert matrices $S$-equivalent to $V$; see \cites{seifert,trotter73}.
%Let $[K]$ denote the $S$-equivalence class of the Seifert matrices of a knot $K$.

  Motivated by the unknotting operation,
the \emph{algebraic unknotting operation}
    assigns a Seifert matrix $W$ to
    $     \begin{pmatrix}
              \varepsilon &0 &0\\
               1 &x &M\\
               0 &N^T &W
           \end{pmatrix}
       $ for $\varepsilon=\pm 1$
       and $x\in \Z$, where $M$ and $N$ are row vectors \cite{mura90}.

%In \cite{saeki}, 
%it is shown that
%there is a natural connection between 
%the algebraic unknotting operation and the unknotting operation.
%Now we focus on the local part of a knot where the crossing change takes place. 
The unknotting operation 
can be seen as adding a twist to a knot, turning Figure~\ref{fig:twist}-a into
Figure~\ref{fig:twist}-b, which is equivalent to \ref{fig:twist}-c.
The twist may fall into two types,
corresponding to
two types of the algebraic unknotting operations.
To distinguish them, set $\varepsilon=1$ for Figure~\ref{fig:twist}-b and
$\varepsilon=-1$ otherwise.
We call the corresponding operation an $\varepsilon$-unknotting operation.
%It is easy to see that Figure~\ref{fig:twist}-a and Figure~\ref{fig:twist}-d
%are of the same knot.
%The Seifert matrix shown in Figure~\ref{fig:twist}-d.
Set $W$ to be a Seifert matrix of Figure~1-d.
We add a twist to the knot in Figure~1-d, so that 
the result is the knot in Figure~1-e or Figure~1-f.
Set $\alpha$ and $\beta$  to be the two new generators of the first homology group of the 
Seifert surfaces in Figure~1-e,f which do not appear in Figure~1-d.
By choosing the direction of $\alpha$ such that $\operatorname{lk}(\alpha,\beta^+)=1$,
we have $\operatorname{lk}(\alpha,\alpha^+)=\varepsilon$ and $\operatorname{lk}(\beta,\beta^+)=x$.
The Seifert matrices of Figure~\ref{fig:twist}-e,f coincide with the result of
the algebraic unknotting operation.

        \begin{figure}[h]
            \centering
            \includegraphics[width=0.32\textwidth]{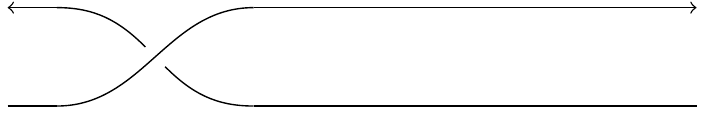}
    \includegraphics[width=0.32\textwidth]{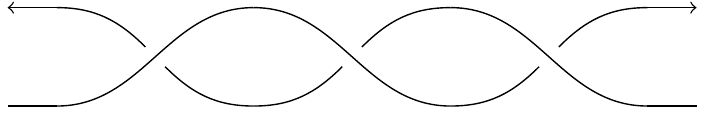}
    \includegraphics[width=0.32\textwidth]{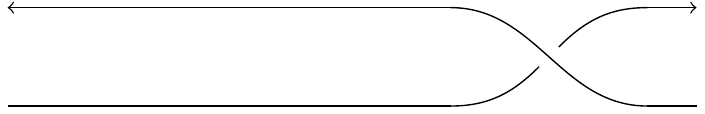}
    \\
    \hspace{0.\textwidth}(a)
    \hspace{0.32\textwidth}(b)
    \hspace{0.3\textwidth}(c)
    \hspace{0.15\textwidth}\\
    ~\\
    \includegraphics[width=0.32\textwidth]{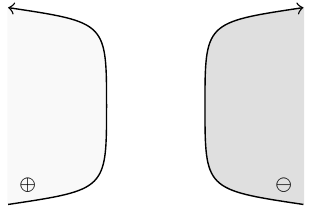}
    \includegraphics[width=0.33\textwidth]{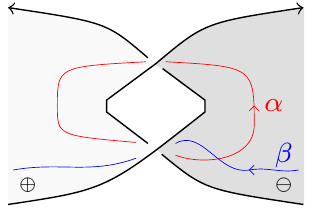}
    \includegraphics[width=0.33\textwidth]{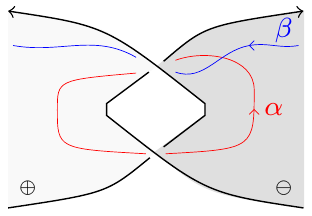}
    \\
    \hspace{0.\textwidth}(d)
    \hspace{0.32\textwidth}(e)
    \hspace{0.3\textwidth}(f)
    \hspace{0.15\textwidth}\\
             \caption{Unknotting operation}
            \label{fig:twist}
        \end{figure}

\emph{The Gordian distance} \cite{mura85}
    between 
 $K$ and $K'$, denoted by $d_G(K,K')$, is the minimal number of crossing changes
 needed to turn $K$ into $K'$.
The \emph{unknotting number} of $K$, denoted by $u(K)$, is defined by $u(K)=d_G(K,O)$,
where $O $ is the trivial knot.
    Let $\mathcal{V}$ and $\mathcal{V'}$ be two $S$-equivalence classes.
    For any two Seifert matrices $V $ and $V'$ such that $V\in\mathcal{V}$ and $V'\in\mathcal{V'}$,
there exist a sequence of algebraic unknotting operations and $S$-equivalences transforming
$V$ to $V'$.
The \emph{algebraic Gordian distance} between $\mathcal{V}$ and $\mathcal{V'}$,
denoted by $d_G^a(\mathcal{V},\mathcal{V'})$,
is the minimal number of algebraic unknotting operations in such a sequence
\cites{mura90,fogel}.

Clearly, $d_G^a(\V,\V')=d_G^a(\V',\V)$ and $d_G^a(\V,\V')=0$
if and only if $\V=\V'$, hence $d_G^a$ defines a metric on the space of $S$-equivalence class
of Seifert matrices.

The following lemma, first proved in \cite{mura90}*{Proposition~2, p.286}, is an immediate consequence of the following observation.
Given any sequence of unknotting operations for a knot, one can construct a
sequence of algebraic unknotting operations for its Seifert matrix.
\begin{lemma}[Murakami]
    For two knots $K_1$ and $K_2$ with Seifert matrices $V_1$ and $V_2$, respectively,
    we have
    $d_G(K_1,K_2)\ge d_G^a([V_1],[V_2])$.
    \label{lem_mura2}
\end{lemma}

The \emph{algebraic unknotting number} $u_a(\mathcal{V})$ is defined to be 
    $d_G^a(\mathcal{V},\mathcal{O})$, where $\mathcal{O}$ is the $S$-equivalence
    class of the $0\times 0$ matrix \cites{mura90}.

    The \emph{Alexander polynomial} of a Seifert matrix $V$, denoted by $\Delta_V$,
    can be calculated by $\Delta_{V}= \det\left(t^{\frac{1}{2}}V-t^{-\frac{1}{2}}V^T\right)$.
    The Alexander polynomial is a knot invariant, which means that any
    two Seifert matrices
    of a given knot have the same Alexander polynomial.
    If $V$ is a Seifert matrix of $K$, we write $\Delta_{K}=\Delta_{V}$.
Saeki proved $\ds u_a([V])=\min_{K_0}d_G(K,K_0)$, where $K_0$ is a knot
with $\Delta_{K_0}=1$; see \cite{saeki}.
We will write $u_a(V) =u_a([V])$ for convenience and set $u_a(K) =u_a(V)$.

Analogously, 
the \emph{Alexander polynomial distance} between two
Alexander polynomials $\Delta$ and $\Delta'$, 
denoted by $\rho(\Delta,\Delta')$,
is defined by
$$    \rho(\Delta,\Delta')=\min d_G(K,K'),$$
where the minimum is taken over all knots $K,K'$ with Alexander polynomials $\Delta, \Delta'$, respectively \cite{kawa12}.
Kawauchi pointed out that $1\le\rho(\Delta,\Delta')\le 2$
for $\Delta, \Delta'$ distinct;
see \cite{kawa12}*{p.954}.
Given any Alexander polynomial $\Delta$,
there exists a knot $K$ with unknotting number one such that $\Delta_K=\Delta$
\cite{kondo}*{Theorem 3, p.558}. It follows that $\rho(\Delta,1)\le 1$ for any $\Delta$,
and hence
$\rho(\Delta,\Delta')\le \rho(\Delta,1)+\rho(\Delta',1)\le 2$ for any pair of Alexander polynomials
$\Delta$ and $\Delta'$.

Kawauchi called the following question Jong's Problem \cite{kawa12}*{p.954} as 
mentioned in Jong's papers \cites{jong1,jong2,jong3}.
\begin{question}
Find Alexander polynomials $\Delta$ and $\Delta'$ such that
$\rho(\Delta,\Delta')= 2$.
\end{question}

Equivalently, this question asks when 
two Alexander polynomials cannot be realized by  knots with Gordian distance one.
In \cite{kawa12}*{Corollary~4.2, p.955}, Kawauchi gives a criterion for this, assuming that both Alexander polynomials $\Delta$ and $\Delta'$ have degree two. 
In Section~4, we will give new criteria in Corollaries~\ref{thm_main3} and \ref{thm_main4} which require that only one of the 
%for two Alexander polynomials with only one Alexander polynomial of degree two.
Alexander polynomials has degree two.

Note that  a question of
Nakanishi  asks if
$\rho(\Delta_{3_1},\Delta_{4_1})= 2$; see \cite{naka}*{p.334}.
It is answered positively by Kawauchi \cites{kawa12}.
Our result gives another method to answer it.

The study of the unknotting number and the Gordian distance is closely  related
to pairing relations of covering spaces.
Let $V$ be a Seifert matrix, i.e. $V$ is a $2n \times 2n$ integral matrix with $\det(V -V^T)=1.$ The \emph{Alexander module}, denoted by $A_V$, is defined
by $A_V=\Lambda^{2n} /(tV-V^T)\Lambda^{2n}$,
where $\Lambda=\Z[t,t^{-1}]$.
If $V$ is a Seifert matrix for  $K$, then $A_V\cong H_1(\tilde{X}(K);\Z)$,
where $\tilde{X}(K)$ is the infinite cyclic cover of the complement of $K$.
If two Seifert matrices,  $V$ and $V'$ are $S$-equivalent, then their Alexander modules
are isomorphic.
%Thus we can denote the Alexander module of $\V$ by $A_{\mathcal{V}}$.
The \emph{Blanchfield pairing} of  $V$ is
a map $\beta: A_V\times A_V\longrightarrow Q(\Lambda)/\Lambda$,
which is 
a sesquilinear form, meaning
$\beta(ax,by)=a\bar{b}\beta(x,y)$, where $\bar{b}= b|_{t=t^{-1}}$ and $Q(\Lambda)$
is the field of fractions of $\Lambda$; see \cite{blanch}.
%The Blanchfield pairing of $V$ is given by the matrix 
%$(t-1)(tV-V^T)^{-1} \pmod{\Lambda}$; see \cites{kearton,trotter62}.
%Fix generators $\left\{ g_1,\dots,g_{2n} \right\}$ for $A_V$ such that 
%$A_V$ is presented by $(tV -V^T )$. The Blanchfield pairing of $g_i$ and $g_j$
%is the $(i,j)$-entry of $(t - 1)(V -tV^T )^{-1}$ modulo $\Lambda$;
%see \cites{friedl,kearton,trotter62}.
By  \cite{friedl}*{Theorem 1.3}, if $K$ is an oriented knot in $S^3$ with Seifert matrix $V$
of size $2n$, then the Blanchfield pairing is isometric to the pairing 
$$\Lambda^{2n} /(tV-V^T)\Lambda^{2n}\times\Lambda^{2n} /(tV-V^T)\Lambda^{2n}\longrightarrow
Q(\Lambda)/\Lambda
$$
given by
$(v,w)\mapsto
 v^T(t - 1)(V -tV^T )^{-1}\bar{w}$ modulo $\Lambda$.

Note that
two matrices have the same Blanchfield pairing structure,
if and only if they are $S$-equivalent
\cite{trotter73}.

There are a lot of papers on how to calculate the Gordian distance for 
two given knots. A large table of the Gordian distances
is given by Moon \cites{moon}.
However, the algebraic Gordian distance of knots is rarely studied. 
We are interested in the restrictions on their $S$-equivalence classes
when their algebraic Gordian distance is one.
We will use these restrictions to provide lower bounds for various distances in knot theory.

We now list some existing results for future use.
To detect the Gordian distance, some lower bound criteria are proved.
The signature criterion
\cites{murasugi,mura85}
is $d_G(K,K')\ge \frac{1}{2}|\sigma(K)-\sigma(K')|$, where $\sigma(K)$ is the
signature of $K$.

Murakami generalized Lickorish's result \cite{lick} on the double branched cover
and showed that
if $u(K)=d_G(K,K')=1$, then 
there exists an integer $d$
such that $$\frac{2d^2}{D(K)}\equiv\pm\frac{D(K)-D(K')}{2D(K)}\pmod{1},$$
where $D(K)$ and $D(K')$ denote the determinants of $K$ and $K'$, respectively
\cite{mura85}. We call it Murakami's obstruction.

As to the algebraic unknotting number,
we refer to the following lemma of Murakami \cite{mura90}*{Theorem~5, p.288}, which we
will use later in Section~4.
\begin{lemma}[Murakami]
\label{lem_mura}
If $u_a(K)=1$, then there exists a generator
$\alpha$ for the Alexander module of $K$
such that the Blanchfield pairing $\beta(\alpha,\alpha)=\pm \frac{1}{\Delta_K}$.
Moreover, the Blanchfield pairing is given by a $1\times 1$-matrix
$(\pm \frac{1}{\Delta_K})$.
\end{lemma}

\section{The Seifert matrix}
In this section, we recall
the definition of the algebraic unknotting operation. We notice that the matrix 
        $ 
        \begin{pmatrix}
              \varepsilon &0 &0\\
              1 &x &M\\
               0 &N^T &W
        \end{pmatrix}
        $
        is not the only possible result of adding a twist.

\begin{lemma}
    Let $W$ be a Seifert matrix of $K$.
If an $\varepsilon$-unknotting operation relates $K$ to $K'$,
then   both  $     \begin{pmatrix}
              \varepsilon &0 &0\\
              \pm 1 &x &M\\
               0 &N^T &W
           \end{pmatrix}
       $ and 
   $     \begin{pmatrix}
              \varepsilon &\pm 1 &0\\
              0 &x &M\\
               0 &N^T &W
           \end{pmatrix}
       $ 
       are Seifert matrices of $K'$.
    \label{prop:1}
\end{lemma}
\begin{proof}
    The Seifert surface of $K'$ can be constructed as in
    Figure~\ref{fig:matrix}-a or Figure~\ref{fig:matrix}-b,
    corresponding to
$     \begin{pmatrix}
              \varepsilon &0 &0\\
              \varsigma &x &M\\
               0 &N^T &W
           \end{pmatrix}
       $ and 
   $     \begin{pmatrix}
              \varepsilon &\varsigma &0\\
              0 &x &M\\
               0 &N^T &W
           \end{pmatrix}
           $, respectively.
           The direction of $\alpha$ determines $\varsigma=1$ or $\varsigma=-1$.
\end{proof}

        \begin{figure}[h]
            \centering
    \includegraphics[width=0.32\textwidth]{./pic/fig9.pdf}
    \hspace{0.1\textwidth}
    \includegraphics[width=0.32\textwidth]{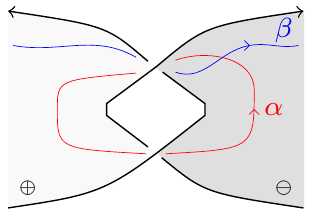}
    \\
    \includegraphics[width=0.32\textwidth]{./pic/fig8.pdf}
    \hspace{0.1\textwidth}
    \includegraphics[width=0.32\textwidth]{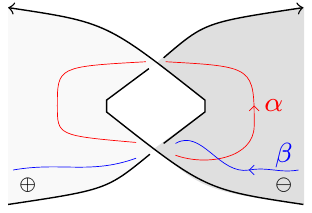}
    \\
    \hspace{0.\textwidth}(a)
    \hspace{0.1\textwidth}
    \hspace{0.285\textwidth}(b)
    %\hspace{0.15\textwidth}
    \\
            \caption{Seifert surfaces}
            \label{fig:matrix}
        \end{figure}

        It is often hard to tell whether two matrices are $S$-equivalent or not,
        especially for matrices of size larger than $2\times 2$.
        As a consequence to Lemma~\ref{prop:1}, we have the following equivalence.
\begin{lemma}
 $     \begin{pmatrix}
              \varepsilon &0 &0\\
              \pm 1 &x &M\\
               0 &N^T &W
           \end{pmatrix}
       $ 
       is $S$-equivalent to 
   $     \begin{pmatrix}
              \varepsilon &\pm 1 &0\\
              0 &x &M\\
               0 &N^T &W
           \end{pmatrix}
           $ for $\varepsilon=\pm 1$.
           \label{prop:2}
\end{lemma}

Now we  show
that some Alexander polynomials are only realizable by Seifert matrices with
algebraic unknotting number one.

\begin{lemma}
    If a $2\times 2$ Seifert matrix $V$ has
    $\det V=D\in \{1,2,3,5\}$, then $u_a(V)=1$.
    \label{lem_c}
\end{lemma}
\begin{proof}
    Since $\det V>0$ and the matrix size is $2\times 2$,
    either $V$ or $-V$ is positive definite.
Every $2\times2$ positive definite Seifert matrix is congruent to a matrix 
$
\begin{pmatrix}
    a&b+1\\b&c
\end{pmatrix}
$, where $0<2b+1\le \min(a,c)$;
see \cite{trotter73}*{p.204}.
Since $b=0$ is the only solution to $ac-b(b+1)=D$, we obtain $ac=D$.
Therefore, we have  either $a=D$ and $c=1$, or $a=1$ and $c=D$.
By Lemma~\ref{prop:2},
$ \begin{pmatrix} 1&1\\0&D \end{pmatrix} $
is $S$-equivalent to
$ \begin{pmatrix} 1&0\\1&D \end{pmatrix} $,
which is congruent to
$ \begin{pmatrix} D&1\\0&1 \end{pmatrix} $.

By Lemma~\ref{prop:1},
both
$
\begin{pmatrix}
    1&1\\0&D
\end{pmatrix}
$ and $
\begin{pmatrix}
    -1&-1\\0&-D
\end{pmatrix}
$  have algebraic unknotting number one, so the proof is complete.
\end{proof}

\begin{lemma}
    For a Seifert matrix $V$,
    if $\Delta_V=ht+ht^{-1}+1-2h$ with $h\in \{1,2,3,5\}$,  then $u_a(V)=1$.
    \label{lem_d}
\end{lemma}
\begin{proof}
    Because $\Delta_V=ht+ht^{-1}+1-2h$,
$V$ is $S$-equivalent to a
$2\times 2$ Seifert matrix $V'$ with $\det V'=h$ ; see \cite{trotter62}*{pp.484-486}.
By Lemma~\ref{lem_c}, 
we have $u_a(V)=1$.
\end{proof}

The next lemma relates the distance between matrices with
the distance between polynomials.
\begin{lemma}
If $K_1$ and $K_2$ are knots with Seifert matrices $V_1$ and $V_2$ and Alexander
polynomials $\Delta_{K_1}$ and $\Delta_{K_2},$ respectively, then
$d_G^a([V_1],[V_2])\ge \rho(\Delta_{K_1},\Delta_{K_2})$.
\label{lem_e}
\end{lemma}
\begin{proof}
    If $d_G^a([V_1],[V_2])=0$,
$V_1$ is $S$-equivalent to $V_2$ and hence
$\Delta_{K_1}=\Delta_{K_2}$, which gives 
$ \rho(\Delta_{K_1},\Delta_{K_2})=0$.

    If $d_G^a([V_1],[V_2])=1$, 
    then we can find Seifert matrices $V_1' \in[V_1]$
    and
    $V_2'\in[V_2]$ such that $V'_1$ can be turned into $V'_2$ by one algebraic unknotting operation.
    We can construct two Seifert surfaces as shown in Figure~1-d,e (or Figure~1-d,f), where Seifert matrices of them are $V'_1$ and $V'_2$ respectively.
Let  $K'_1$ and $K'_2$  be the boundaries of these two Seifert surfaces.
Then we have $d_G(K'_1,K'_2 )=1$. Since 
 $\Delta_{K'_1}=\Delta_{K_1}$ and $\Delta_{K'_2}=\Delta_{K_2}$,
we obtain 
$\rho(\Delta_{K_1},\Delta_{K_2})\le 1$.

    %If $d_G^a([K_1],[K_2])=1$, by
%the natural connection between the algebraic unknotting operation and
%the unknotting operation,
%there exist $K'_1$ and $K'_2$ such that 
 %their Seifert matrices are in $[K_1]$   and $[K_2]$ respectively with
%$d_G(K'_1,K'_2)=1$. Then we obtain
%$\Delta_{K'_1}=\Delta_{K_1}$, $\Delta_{K'_2}=\Delta_{K_2}$ and consequently
%$\rho(\Delta_{K_1},\Delta_{K_2})=1$.
If $d_G^a([V_1],[V_2])\ge 2$, the inequality holds because $\rho(\Delta_1,\Delta_2)\le 2$
for any pair of Alexander polynomials $\Delta_1,\Delta_2$.
\end{proof}

Consequently, we have
$d_G(K_1,K_2)\ge d_G^a([V_1],[V_2])\ge \rho(\Delta_{K_1},\Delta_{K_2})$.

\section{Main theorem and its consequence}
In this section, we examine the structure of the Blanchfield pairing
realized by a pair of $S$-equivalence classes with algebraic
Gordian distance one. We obtain conditions 
expressed in terms of the Alexander polynomials for two $S$-equivalence classes to have
algebraic Gordian distance one.
In corollaries to the main theorem,
we provide an answer to Jong's question by showing that two Alexander polynomials cannot be 
realized by knots with Gordian distance one unless
a corresponding quadratic equation admits an integer solution.

\begin{theorem}[Main Theorem]
    Let $V$ and $V'$ be two Seifert matrices.
    If the algebraic Gordian distance $d_G^a([V],[V'])$=1,
then there exist $a\in A_V$
and $a'\in A_{V'}$ such that
$\ds\beta(a,a)\equiv \pm\frac{\Delta_{ V' }}{\Delta_V} \pmod{\Lambda}$
and
$\ds\beta(a',a')\equiv \pm\frac{\Delta_V}{\Delta_{V'}} \pmod{\Lambda}$.
    \label{thm_main2}
\end{theorem}

\begin{proof}
    If $[V]$ and $[V']$ have algebraic Gordian distance one,
    there exist $W\in[V]$ and $W'\in[V']$ such that
    $W$ can be obtained from $W'$ by an algebraic unknotting operation.
    By definition,
    the algebraic unknotting operation assigns $W'$ to 
    $     \begin{pmatrix}
              \varepsilon &0 &0\\
               1 &x &M\\
               0 &N^T &W'
           \end{pmatrix}
           $ for $\varepsilon=\pm 1$.
           Therefore, we have
$$
    W-tW^T=
    \begin{pmatrix}
        \varepsilon(1-t) &-t &0\\
               1 &x (1-t)&M-tN\\
               0 &N^T-tM^T &W'-tW'^{T}
           \end{pmatrix}.
$$

    Let $a_1$ be the first element of the basis for $A_V$
    so that
    the Blanchfield pairing
    $\beta(a_1,a_1)$ is the $( 1,1 )$-entry of matrix
    $(t-1)(W-tW^T)^{-1}$ modulo $\Lambda$.
    The inverse of a non-singular matrix $M$
    %can be calculated by its adjugate matrix, so we have
    is equal to $\ds\frac{\operatorname{adj}M}{\det M}$,
    where $\operatorname{adj}M$ is the adjugate matrix, i.e. the transpose of the 
    cofactor matrix of $M$.

\begin{align}
    \displaystyle
    \beta(a_1,a_1)
    &\equiv
    (t-1)
    \frac{\left[\operatorname {adj} (W-tW^T) \right]_{ 1,1}
        }{ \det(W-tW^T) }\\
    &\equiv(t-1)\frac{
    \det \begin{bmatrix}
x (1-t)&M-tN\\
N^T-tM^T &W'-tW'^{T}
    \end{bmatrix}
    }{\det(W-tW^T) 
}.
\end{align}

The Alexander polynomials
are given by
\begin{align}
 \Delta_{V}&=\Delta_{W}=t^{-g} \det(W-tW^T) \\
 \Delta_{V'}&=\Delta_{W'}=t^{1-g} \det(W'-tW'^T),
\end{align}
where $2g$ 
is the size of $W$.
The determinant
\begin{equation}
    \label{e:4}
\det(W-tW^T)=\varepsilon(1-t)
\det \begin{bmatrix}
x (1-t)&M-tN\\
N^T-tM^T &W'-tW'^{T}
\end{bmatrix}
+t\det(W'-tW'^T)
.
\end{equation}
By substituting (3), (4), (5) from (2),
we have
\begin{equation*}
    \beta(a_1,a_1)\equiv
    \frac{\varepsilon(\Delta_{V'} -\Delta_{V} )}{\Delta_V}
    \equiv\varepsilon \frac{\Delta_{V'} }{\Delta_V}
    \pmod{\Lambda}.
\end{equation*}
Let $a'_1$ be the first element of the basis for $A_{V'}$.
The equation of $\beta(a'_1,a'_1)$ can be derived in the same way.
\end{proof}
Theorem~\ref{thm_main2} gives a condition on the Blanchfield pairing when the 
algebraic
Gordian distance is one. Using similar methods, we deduce the following corollary,
which gives the same obstruction as above for a pair of knots with Gordian distance one.

\begin{corollary}
    \label{col:gordian}
If $K$ and $K'$ are two knots with
$d_G(K,K')=1$, then there exist $a\in H_1(\tilde{X}(K))$
and $a'\in H_1(\tilde{X}(K'))$ such that
$\ds\beta(a,a)\equiv \pm\frac{\Delta_{ K' }}{\Delta_K} \pmod{\Lambda}$
and
$\ds\beta(a',a')\equiv \pm\frac{\Delta_K}{\Delta_{K'}} \pmod{\Lambda}$.
\end{corollary}

As the Blanchfield pairing
is a complicated form, we now focus on the case where the Alexander module
is cyclic. We prove the following corollary and show that it improves
existing results.

%We recall Corollary~\ref{col:alg}:

\begin{corollary}
    If $u_a([V])=d_G^a([V],[V'])=1$,
then there exists $c\in \Lambda$ such that 
$\pm \Delta_{V'}\equiv c\bar{c}\pmod{\Delta_V}$.
\label{col:alg}
\end{corollary}

\begin{proof}
    Since $u_a([V])=d_G^a([V],[V'])=1$, by Theorem~\ref{thm_main2}, there exist
    $a\in A_V$ and $g\in A_{V}$
such that 
$$
\beta(a,a)\equiv \pm\frac{\Delta_{ V' }}{\Delta_V}\pmod{ \Lambda}
\text{~~~~~and~~~~~}
\beta(g,g)\equiv \pm\frac{1}{\Delta_{V}} \pmod{ \Lambda}.
$$

By Lemma~\ref{lem_mura}, the Blanchfield pairing on $A_V$
is cyclic and generated by $g$.
Therefore, there exists $c\in \Lambda$ such that $a=cg$.
Hence we have
$$\pm \frac{\Delta_{V'}}{\Delta_V} \equiv\beta(cg,cg)
\equiv c\bar{c}\beta(g,g)\equiv\frac{ c\bar{c}}{ \Delta_V}  \pmod{ \Lambda},$$
which gives $\pm \Delta_{V'}=c\bar{c}\pmod{\Delta_V}$. This completes the proof.
\end{proof}
The following result is a natural consequence of Corollary \ref{col:alg}.
\begin{corollary}
    If $u_a(K)=d_G(K,K')=1$, then there exists $c\in \Lambda$ such that 
$\pm \Delta_{K'}\equiv c\bar{c}\pmod{\Delta_K}$.
\label{col:1}
\end{corollary}
\begin{proof}
    Let $V$ and $V'$ be Seifert matrices for $K$ and $K'$, respectively.
    Clearly  $\Delta_V=\Delta_{K}$ and $\Delta_{V'}=\Delta_{K'}$.
    Since $d_G^a([V],[V'])\le d_G(K,K')=1$, we have $d_G^a([V],[V'])\le1$.

    If $d_G^a([V],[V'])=0$, then $[V]=[V']$ and
    $\Delta_K=\Delta_{K'}$, and the result holds by taking $c=0$.

    If $d_G^a([V],[V'])=1$, then by Corollary~\ref{col:alg} 
    there exists $c\in\Lambda $  such that
$\pm \Delta_{K'}\equiv c\bar{c}\pmod{\Delta_K}$.
\end{proof}

\begin{remark}
It is worth mentioning that
Corollary~\ref{col:alg} implies Kawauchi's result
Theorem~\ref{thm:kawa}.
%, which is a result of Kawauchi \cite{kawa12}*{Theorem~1.2, p.949}.
Note that there are infinitely many knots with trivial Alexander polynomial.
Since $u(K)=1$ is a special case of $u_a(K)=1$,
Corollary~\ref{col:1} implies Theorem~\ref{thm:kawa}.
%Moreover, Corollary~\ref{col:alg} concerns the algebraic Gordian distance.
%By the inequality $d_G(K_1,K_2)\ge d_G^a([K_1],[K_2])$,
%%the algebraic Gordian distance subsumes the lower bound of the Gordian distance.
%cases satisfying the algebraic assumption does not necessarily satisfy the geometric
%assumption.
However, since Corollary~\ref{col:1} follows from Corollary~\ref{col:alg},
we see that Corollary~\ref{col:alg} also implies Theorem~\ref{thm:kawa}.
In Section~5, we will present an example to show the strength of our approach.
Let $K_1$ and $K_2$ be knots
with the same Alexander polynomials
as $3_1$ and $9_{25}$, respectively.
By Lemma~\ref{lem_d}, we have $u_a(K_1)=1$,
so  Corollary~\ref{col:alg} applies to show that $d_G(K_1,K_2)\ge d_G^a([V_1],[V_2])\ge 2$.
However, Theorem~\ref{thm:kawa} does not apply in this case,
because we do not necessarily have $u(K_1)=1$.
\end{remark}

Our next aim is to give new solutions to Jong's problem
by finding Alexander polynomials $\Delta$ and $\Delta'$ such that
$\rho(\Delta,\Delta')= 2$.

\begin{corollary}
    Suppose $V$ and $V'$ are Seifert matrices with Alexander polynomials
    $\Delta_V$ and $\Delta_{V'}$, respectively, such that
    $\Delta_V=h(t+t^{-1})+1-2h$, where $|h|$ is prime or $1$, and 
    $\Delta_{V'}\equiv d \pmod{\Delta_{V}}$ for some $0\neq d\in \Z$. 
    If $u_a([V])=1$
    and if
    $h^2x^2+y^2+(2h-1)xy=\pm d $ does not have an integer solution,
    then the algebraic Gordian distance $d^a_G([V],[V'])> 1$.
    \label{thm_main3}
\end{corollary}
\begin{proof}
    Seeking a contradiction, suppose $d^a_G([V],[V'])=1$.
    By Corollary~\ref{col:1}, there exists $c\in\Lambda$
    such that $c\bar{c}\equiv \pm \Delta_{V'}\equiv \pm d\pmod{\Delta_V}$.
    Let
    $$
c=\sum_{-n\le i\le m}a_it^i
    \text{~~~~~~~~and~~~~~~~~}
    \bar{c}=\sum_{-m\le i\le n}a_{-i}t^i,
    $$
    which gives
    $$
    c\bar{c}=
    a_{-n}a_{m}t^{m+n}+\dots+a_{m}a_{-n}t^{-(m+n)}
    .
    $$

    If $c$ can be expressed as $c= pt^{j+1}+qt^j$, where $p$ and $q$ are integers,
     we have $ (p^2+q^2) + pq(t+t^{-1})\equiv \pm d\pmod{\Delta_V} $.
     Since $|h|$ is prime or $1$,  either $h|p$ or $h|q$.
     Without loss of generality, we may assume $p=hx$.
By substituting $p$, we obtain that $ h^2x^2+q^2+(2h-1)xq=\pm d $.
If it does not have an integer solution, the algebraic Gordian distance 
must be greater than one.

    If $c$ has more than two terms,
    we have
    $h|a_{-n}a_{m}$ follows from
    $c\bar{c}\equiv\pm d\pmod{\Delta_V}$,
    so
    either $h|a_{-n}$ or $h|a_{m}$.
    Hence we obtain
    $$\ds
    c\equiv\sum_{1-n\le i\le m}a'_it^i     \text{~~~~~or~~~~~}
    \sum_{-n\le i\le m-1}a'_it^i\pmod{\Delta_V}
    ,$$
    where $\left\{ a'_i \right\}$ are integer coefficients.
    Repeat this step until we deduce $c\equiv pt^{j+1}+qt^j\pmod{\Delta_V}$, where 
    $p$ and $q$ are integers.
    The rest of the proof follows in the same manner.
    \end{proof}
The following corollary is an immediate consequence.

\begin{corollary}
    The Alexander polynomial distance
    $\rho(t-1+t^{-1},\Delta)=2$ if
$\Delta\equiv  2+4m \pmod{t-1+t^{-1}}$
for some $m\in \Z$.
    \label{thm_main4}
\end{corollary}
    \begin{proof}
        By Lemma~\ref{lem_d}, any knot with
        Alexander polynomial $t-1+t^{-1}$ has algebraic unknotting number one.
        By Corollary~\ref{thm_main3}, it suffices to show that
        $x^2+y^2+xy= 2+4m$ does not have an integer solution.

    Now we check the parities of $x$ and $y$.
            If $x$ and $y$ are 
            both odd
           or one is even and the other is odd, then 
            $x^2+y^2+xy$ is odd, which is a contradiction.
            Otherwise, if $x$ and $y$ are both even, then $x^2+y^2+xy\equiv 0 \pmod{4}$,
which is also a contradiction.
Hence the proof is complete.
\end{proof}

\begin{remark}
There are other applications of Corollary~\ref{thm_main3}.
For example, by Lemma~\ref{lem_d},  Corollary~\ref{thm_main3} gives the same result for
$\rho(\Delta,ht+ht^{-1}+1-2h)=2$ with $h\in \{1,2,3,5\}$ provided that
 $h^2x^2+y^2+(2h-1)xy=\pm d $ does not have an integer solution.
\end{remark}

\begin{remark}
Corollary~\ref{thm_main4}
offers another route to answer Nakanishi's question \cite{naka}*{p.334}, which asks if 
$\rho(\Delta_{3_1},\Delta_{4_1})= 2$.
Moreover,
it implies that any
two Seifert matrices with
Alexander polynomials same as $3_1$ and $4_1$, respectively, cannot be
turned into each other by one algebraic unknotting operation.
\end{remark}

\section{Example}
Moon computed the Gordian distances between many knots, which he listed in a table.
By Moon's results, the lower bound of $d_G(3_1,9_{25})$ is one.
By our method, we now prove that any pair of knots with 
the same Alexander polynomial as $3_1$ and $9_{25}$, respectively, cannot have 
Gordian distance one.
Therefore, the inequality $d_G(3_1,9_{25})\ge 2$ holds. % as a special case.
Moreover, we will deduce that the algebraic Gordian distance $d^a_G([V_1],[V_2])=2$,
where $V_1$ and $V_2$ are Seifert matrices for $3_1$ and $9_{25},$ respectively.
%We refer to the diagrams 
%The following diagrams are from \cite{knotinfo}.
\begin{figure}[h]
    \hspace{0.05\textwidth}
    \includegraphics[width=0.25\textwidth]{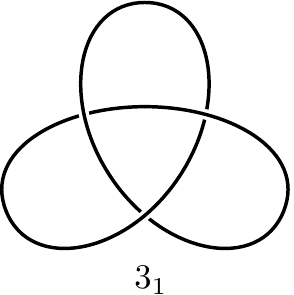}
    \hspace{0.2\textwidth}
    \includegraphics[width=0.28\textwidth]{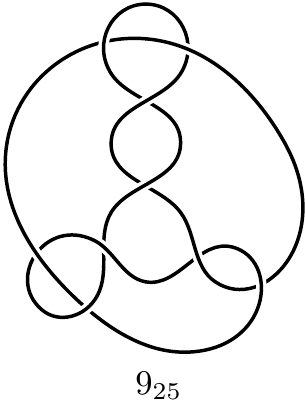}
    \caption{Knot diagrams for $3_1$ and $9_{25}$.}
    
%    \vspace{-0.5cm}
%    \begin{align*}
%\hspace{0.1\textwidth}
%\Delta_{3_1}&=t+t^{-1}-1
%\hspace{0.1\textwidth}
%&\Delta_{9_{25}}&=-3 t^2-3t^{-2}+12 t+12t^{-1}-17
%\\
%\hspace{0.1\textwidth}
%\sigma(3_1)&=-2 
%\hspace{0.1\textwidth}
%&\sigma(9_{25})&=-2 
%\\
%\hspace{0.1\textwidth}
%D(3_1)&=3 
%\hspace{0.1\textwidth}
%&D(9_{25})&=47 
%    \end{align*}
%\caption{
%    $3_1$  and $9_{25}$}
    \label{fig:knots}
\end{figure}

Consider the knot diagrams for $3_1$ and $9_{25}$ in Figure \ref{fig:knots}.
From \cite{knotinfo} we have 
\begin{align*}
\hspace{0.1\textwidth}
\Delta_{3_1}&=t+t^{-1}-1
\hspace{0.05\textwidth}
&\Delta_{9_{25}}&=-3 t^2-3t^{-2}+12 t+12t^{-1}-17
\\
\hspace{0.1\textwidth}
\sigma(3_1)&=-2 
\hspace{0.05\textwidth}
&\sigma(9_{25})&=-2 
\\
\hspace{0.1\textwidth}
D(3_1)&=3 
\hspace{0.05\textwidth}
&D(9_{25})&=47 
    \end{align*}

%$\Delta_{3_1}=t+t^{-1}-1$
% and
Since $\Delta_{9_{25}}= (-3t+9-3t^{-1}) \Delta_{3_1}-2
$,
Corollary~\ref{thm_main4} applies to show that
$\rho(\Delta_{3_1},\Delta_{9_{25}})=2$. 
From Lemma~\ref{lem_e},
we have
$$ d_G(K_1,K_2)\ge d_G^a([ K_{1} ],[ K_{2} ])\ge
\rho(\Delta_{3_1},\Delta_{9_{25}})=2$$
for any pair of knots $K_1$ and $K_2$ with
$\Delta_{K_{1}}=\Delta_{3_{1}}$ and $\Delta_{K_{2}}=\Delta_{9_{25}}$.
%Then we obtain 
 %$d_G(9_{25},3_1)\ge d_G^a([9_{25}],[3_1])\ge\rho(\Delta_{9_{25}},\Delta_{3_1})= 2 $.

Note that $d_G(3_1,9_{25})\ge 2$ can also be proved by Kawauchi's Theorem~\ref{thm:kawa}, which
is a result of Kawauchi. This is because $3_1$ satisfies the geometric assumption
that $u(3_1)=1$. However, for any pair of knots $K_1$ and $K_2$ with 
the same Alexander polynomial as $3_1$ and $9_{25}$ respectively, 
Theorem~\ref{thm:kawa} does not apply because we may not have that $u(K_i)=1$ for
$i=1,2$. For example, $3_1\#11n_{34}$ and $9_{25}$.

%A concrete example is $d_G(3_1\#K'',9_{25}\#K'')$, where $K''$ is the Whitehead 
%double of $3_1$. Any composite knot has unknotting number at least two so 
%Theorem~\ref{thm:kawa} cannot be applied here.

Moreover, this example demonstrates how our result helps in calculating the algebraic Gordian
distance of  two given $S$-equivalent classes. 
We know $u_a(9_{25})=u_a(3_1)=1$; see \cites{knotorious}. It gives that
$d_G^a([V_1],[V_2])\le u_a(3_1)+u_a(9_{25})=2$, where $V_1$ and $V_2$ are Seifert matrices for
$3_1$ and $9_{25},$ respectively.
Therefore, we have
$d_G^a([V_1],[V_2])=2$.

It is worth mentioning that Murakami's method
does not apply here.
Following Murakami's method, we would have to prove there does not exist an integer $d$
such that $$\frac{2d^2}{D(3_1)}\equiv\pm\frac{D(3_1)-D(9_{25})}{2D(3_1)}\pmod{1}.$$
In fact, any integer
$d$, such that $d\not\equiv 0\pmod{3}$,
satisfies this requirement, which means Murakami's method does not 
work in this case.

Meanwhile, the knot signature criterion fails as well in 
this case.
Since $\sigma(3_1)=\sigma(9_{25})=-2$,
the signature criterion cannot determine whether 
$d_G(K,K')$ is one or not.

%We give another example on calculation of the Alexander polynomial distance of a 
%pair of Alexander polynomials of degree two.
%Set $A_h=h(t+t^{-1})-2h+1$.
%We have $A_{-20}\equiv 5 \pmod{A_5}$.
%Since equation $25x^2+y^2+9xy=5$ has no integer solution,
%by Corollary~\ref{thm_main3} we have $\rho(A_{-20},A_5)=2$.
%Note this cannot be given by Kawauchi's method in [8, Corollary~4.2, p.955].

\section*{Acknowledgements}
    The author wishes to give her many thanks to Prof. Hans Boden,  Prof. Akio Kawauchi
    and Prof. Hitoshi Murakami, for their valuable suggestion.
%She also would like to thank the
%referee for reviewing this paper and offering many valuable comments.

\Addresses

\begin{bibdiv}
\begin{biblist}


\bib{blanch}{article}{
    AUTHOR = {Blanchfield, R. C.},
     TITLE = {Intersection theory of manifolds with operators with
              applications to knot theory},
   JOURNAL = {Ann. of Math. (2)},
    VOLUME = {65},
      YEAR = {1957},
     PAGES = {340--356},
      ISSN = {0003-486X},
}

\bib{knotorious}{webpage}{
    AUTHOR = {Borodzik, M.},
    AUTHOR = {Friedl, S.},
     TITLE = {Knotorious},
       URL = {https://www.mimuw.edu.pl/~mcboro/knotorious.php},
       Date = {2017, July 18},
      %YEAR = {2017},
}

\bib{knotinfo}{webpage}{
    AUTHOR = {Cha, J. C.},
    AUTHOR = {Livingston, C.},
     TITLE = {KnotInfo: Table of Knot Invariants},
       URL = {http://www.indiana.edu/~knotinfo},
       Date = {2017, July 18},
      %YEAR = {2017},
}

\bib{fogel}{thesis}{
    author={Fogel, M. E.},
   title={Knots with algebraic unknotting number one},
   journal={Pacific J. Math.},
   volume={163},
   date={1994},
   number={2},
   pages={277--295},
   issn={0030-8730}
}


\bib{friedl}{article}{
    author={Friedl, S.},
    author={Powell, M.},
   title={A calculation of Blanchfield pairings of 3-manifolds and knots},
   journal={Mosc. Math. J.},
   volume={17},
   date={2017},
   number={1},
   pages={59--77},
   issn={1609-3321},
}




%\bib{gordon}{article}{
    %AUTHOR = {Gordon, C. McA.},
    %AUTHOR = {Litherland, R. A.},
     %TITLE = {On the signature of a link},
   %JOURNAL = {Invent. Math.},
    %VOLUME = {47},
      %YEAR = {1978},
    %NUMBER = {1},
     %PAGES = {53--69},
      %ISSN = {0020-9910},
%}

\bib{jong3}{article}{
    AUTHOR = {Ichihara, K.},
    AUTHOR = {Jong, I. D.},
     TITLE = {Gromov hyperbolicity and a variation of the {G}ordian complex},
   JOURNAL = {Proc. Japan Acad. Ser. A Math. Sci.},
    VOLUME = {87},
      YEAR = {2011},
    NUMBER = {2},
     PAGES = {17--21},
      ISSN = {0386-2194},
}
\bib{jong1}{article}{
    AUTHOR = {Jong, I. D.},
     TITLE = {Alexander polynomials of alternating knots of genus two},
   JOURNAL = {Osaka J. Math.},
    VOLUME = {46},
      YEAR = {2009},
    NUMBER = {2},
     PAGES = {353--371},
      ISSN = {0030-6126},
}
\bib{jong2}{article}{
    AUTHOR = {Jong, I. D.},
     TITLE = {Alexander polynomials of alternating knots of genus two {II}},
   JOURNAL = {J. Knot Theory Ramifications},
    VOLUME = {19},
      YEAR = {2010},
    NUMBER = {8},
     PAGES = {1075--1092},
      ISSN = {0218-2165},
}
\bib{kawa12}{article}{
    AUTHOR = {Kawauchi, A.},
     TITLE = {On the {A}lexander polynomials of knots with {G}ordian
              distance one},
   JOURNAL = {Topology Appl.},
    VOLUME = {159},
      YEAR = {2012},
    NUMBER = {4},
     PAGES = {948--958},
      ISSN = {0166-8641},
}
\bib{kearton}{article}{
    AUTHOR = {Kearton, C.},
     TITLE = {Cobordism of knots and {B}lanchfield duality},
   JOURNAL = {J. London Math. Soc. (2)},
    VOLUME = {10},
      YEAR = {1975},
    NUMBER = {4},
     PAGES = {406--408},
      ISSN = {0024-6107},
}

\bib{kondo}{article}{
    author={Kondo, H.},
   title={Knots of unknotting number $1$\ and their Alexander polynomials},
   journal={Osaka J. Math.},
   volume={16},
   date={1979},
   number={2},
   pages={551--559},
   issn={0030-6126},
}

\bib{levine}{article}{
   author={Levine, J.},
   title={The role of the Seifert matrix in knot theory},
   conference={
      title={Actes du Congr\`es International des Math\'ematiciens},
      address={Nice},
      date={1970},
   },
   book={
      publisher={Gauthier-Villars, Paris},
   },
   date={1971},
   pages={95--98},
}
\bib{lick}{article}{,
    AUTHOR = {Lickorish, W. B. R.},
     TITLE = {The unknotting number of a classical knot},
   Journal = {Contemp. Math.},
    VOLUME = {44},
     PAGES = {117--121},
      YEAR = {1985},
}

\bib{moon}{thesis}{
    AUTHOR = {Moon, H.},
     TITLE = {Calculating knot distances and solving tangle equations
              involving montesinos links},
      NOTE = {Thesis (Ph.D.)--The University of Iowa},
 PUBLISHER = {ProQuest LLC, Ann Arbor, MI},
      YEAR = {2010},
     PAGES = {154},
      ISBN = {978-1124-44226-6},
}
\bib{mura85}{article}{
    AUTHOR = {Murakami, H.},
     TITLE = {Some metrics on classical knots},
   JOURNAL = {Math. Ann.},
    VOLUME = {270},
      YEAR = {1985},
    NUMBER = {1},
     PAGES = {35--45},
      ISSN = {0025-5831},
}
\bib{mura90}{article}{
    AUTHOR = {Murakami, H.},
     TITLE = {Algebraic unknotting operation},
   JOURNAL = {Questions Answers Gen. Topology},
    VOLUME = {8},
      YEAR = {1990},
    NUMBER = {1},
     PAGES = {283--292},
      ISSN = {0918-4732},
}

\bib{murasugi}{article}{
    author={Murasugi, K.},
   title={On the signature of links},
   journal={Topology},
   volume={9},
   date={1970},
   pages={283--298},
   issn={0040-9383},
   %review={\MR{0261585}},
}
\bib{naka}{article}{
    AUTHOR = {Nakanishi, Y.},
     TITLE = {Local moves and {G}ordian complexes. {II}},
   JOURNAL = {Kyungpook Math. J.},
    VOLUME = {47},
      YEAR = {2007},
    NUMBER = {3},
     PAGES = {329--334},
      ISSN = {1225-6951},
}

%\bib{rolf}{book}{
    %AUTHOR = {Rolfsen, D.},
     %TITLE = {Knots and Links (2nd ed.)},
      %NOTE = {Mathematics Lecture Series, No. 7},
 %PUBLISHER = {Publish or Perish, Inc., Berkeley, Calif.},
      %YEAR = {1990},
     %PAGES = {439},
%}

\bib{saeki}{article}{
    AUTHOR = {Saeki, O.},
     TITLE = {On algebraic unknotting numbers of knots},
   JOURNAL = {Tokyo J. Math.},
    VOLUME = {22},
      YEAR = {1999},
    NUMBER = {2},
     PAGES = {425--443},
      ISSN = {0387-3870},
}

\bib{seifert-a}{article}{
    author={Seifert, H.},
       title={\"{U}ber das {G}eschlecht von {K}noten},
        date={1935},
        ISSN={0025-5831},
     journal={Math. Ann.},
      volume={110},
      number={1},
       pages={571\ndash 592},
         url={https://doi.org/10.1007/BF01448044},
      %review={\MR{1512955}},
}

\bib{seifert}{article}{
    AUTHOR = {Seifert, H.},
     TITLE = {Die Verschlingungsinvarianten der zyklischen
              {K}noten\"uberlagerungen},
   JOURNAL = {Abh. Math. Sem. Univ. Hamburg},
    VOLUME = {11},
      YEAR = {1935},
    NUMBER = {1},
     PAGES = {84--101},
      ISSN = {0025-5858},
}


\bib{trotter62}{article}{
    AUTHOR = {Trotter, H. F.},
     TITLE = {Homology of group systems with applications to knot theory},
   JOURNAL = {Ann. of Math. (2)},
    VOLUME = {76},
      YEAR = {1962},
     PAGES = {464--498},
      ISSN = {0003-486X},
}

\bib{trotter73}{article}{
    AUTHOR = {Trotter, H. F.},
     TITLE = {On {$S$}-equivalence of {S}eifert matrices},
   JOURNAL = {Invent. Math.},
    VOLUME = {20},
      YEAR = {1973},
     PAGES = {173--207},
      ISSN = {0020-9910},
}


\end{biblist}
\end{bibdiv}
\end{document}